\documentclass[letterpaper, 10 pt, conference]{ieeeconf}  

\usepackage{amsmath, amssymb, mathtools}
\usepackage{xcolor}
\usepackage{url}
\usepackage{graphicx}
\usepackage{adjustbox}

\usepackage{algorithm}

\newtheorem{theorem}{Theorem}
\newtheorem{proposition}{Proposition}

\newtheorem{assumption}{Assumption}

\newcommand{\R}{\ensuremath{\mathbb R}}  
\newcommand{\N}{\ensuremath{\mathbb N}}  
\renewcommand{\k}{\mathsf{k}}

\newcommand{\F}{F^\varepsilon}

\newcommand{\V}{V^\varepsilon}
\newcommand{\J}{J^\varepsilon}
\newcommand{\B}{B^\varepsilon}
\newcommand{\LF}{L_{F^\varepsilon_\mathrm{y}}}
 
\newcommand{\Y}{Y^\varepsilon}

\usepackage{comment}
\usepackage{lipsum}

\usepackage{cite}

\newcommand{\karl}[1]{\textcolor{black}{#1}}

\usepackage{etoolbox}
\newbool{arxiv}
\setbool{arxiv}{true}

\IEEEoverridecommandlockouts                              

\title{\LARGE \bf
Exponential stability of data-driven nonlinear 
MPC based on input/output models
}

\author{Lea Bold$^{1}$, Irene Schimperna$^{2}$, Karl Worthmann$^{1}$, Johannes Köhler$^{3}$
\thanks{%
    $^{1}$L.\,Bold and K.\,Worthmann are with the Optimization-based Control Group, Institute of Mathematics, TU Ilmenau, Ilmenau, Germany. {\tt\small [lea.bold, karl.worthmann]@tu-ilmenau.de}. L.\,Bold and K.\,Worthmann gratefully acknowledge funding by the German Research Foundation (DFG; project numbers $545246093$ and 535860958).}%
    \thanks{$^{2}$I.\,Schimperna is with the Department of Electrical, Computer and Biomedical Engineering, University of Pavia, Pavia, Italy. {\tt\small irene.schimperna01@universitadipavia.it}}    \thanks{$^{3}$J.\,Köhler is with the Department of Mechanical Engineering, Imperial College London, London, UK. {\tt\small j.kohler@imperial.ac.uk}}%
}

\begin{document}

\maketitle

\begin{abstract}
We consider nonlinear model predictive control (MPC) schemes \karl{without stabilizing terminal conditions, where the model used in the optimization step is generated} 
based on input-output data only.  
We establish exponential stability for sufficiently long prediction horizons assuming exponential stabilizability and a proportional error bound. 
Moreover, we verify the imposed condition on the approximation using kernel interpolation and demonstrate the practical applicability to nonlinear systems \karl{by numerical simulations}. 
\end{abstract}


\section{Introduction}
Model predictive control (MPC) is nowadays a well-established advanced control technique, where the input is determined by solving a finite-horizon constrained optimal control problem~\cite{grune2017nonlinear}. 
For the stability analysis, 
terminal conditions are often utilized\karl{, see~\cite{RawlMayn17} and the references therein}. 
Alternatively, stability can be established 
using a sufficiently long prediction horizon and some stabilizability condition related to the stage cost~\cite{GrimMess05,GrunPann10,CoroGrun20,granzotto2020finite}\karl{, see \cite{grune2017nonlinear} for an overview}.
These results \karl{rely on} 
positive-definite stage cost penalizing the deviation 
from the desired set point. 
However, in \karl{practice}, 
output weighting is often preferred \karl{since state measurements might not be available neccessitating the use of input/output models. Moreover, output weighting} 
is closely related to closed-loop performance. 
\karl{Then}, 
the resulting stage cost is only positive semi-definite in the system state, and the stability analysis requires more general tools, based on \karl{detectability conditions on the stage cost}
~\cite{kohler2021constrained,KohlZeil23}.

MPC relies on a reliable and accurate model, which explains the recent research interest in data-driven approaches~\cite{berberich2025overview}.
To this end, a plethora of 
modeling techniques has been employed in MPC. 
For linear system, an effective method is based on the use of Willems' fundamental lemma~\cite{coulson2021distributionally}, \karl{which allows to directly solve the optimization problem in MPC based on non-parametric models using input/output data only}. 
For nonlinear systems, the explored approaches include Gaussian processes~\cite{scampicchio2025gaussian}, Koopman operator theory~\cite{korda2018linear,strasser2026overview}, neural networks~\cite{salzmann2023real}, and many more. 
Despite the availability of many different modeling techniques, data-driven models come with the presence of model-plant mismatch, which may impact the stability properties of the MPC scheme. 
To fill this gap, recent works have derived conditions, in which data-driven MPC preserves stability \karl{and closed-loop performance}. 
Considering MPC without terminal conditions, \cite{schimperna2025data} shows that asymptotic stability can be obtained in presence of proportional bounds on the modeling error, and proposes a \karl{framework} 
based on Koopman operator theory to \karl{generate data-driven} 
models satisfying the required bounds. 
Similar results have been derived for MPC with terminal conditions in~\cite{KuntRawl25}, considering models with parametric uncertainty. 
Finally, exponential stability of Koopman MPC with terminal conditions has been studied in~\cite{shang2025exponential,schimperna2025stability}, where the latter also introduces a constraint-tightening approach to guarantee robust constraint satisfaction despite model-plant mismatch.

These existing stability results~\cite{schimperna2025data,KuntRawl25,shang2025exponential,schimperna2025stability} rely on state-space models \karl{and, thus, require 
state measurements}. 
However, in many practical applications, only the system outputs are measurable, and, thus, MPC is designed on the base of input/output models. 
The goal of this paper, is to extend the results of asymptotic stability of MPC in presence of approximation errors to models in input/output form. 
In particular, we consider an MPC formulation \karl{with output-weighting stage cost, but} without terminal conditions, in which stability is studied relying on a sufficiently long optimization horizon and \karl{a cost detectability property~\cite{KohlZeil23}}. 
We show that exponential stability of the data-driven MPC closed loop can be achieved if the surrogate model satisfies proportional error bounds \karl{and} 
the system is exponential stabilizable. 
Further, we show that kernel interpolation is a suitable learning technique to provide data-driven surrogate models satisfying \karl{our requirements}. 
Finally, we demonstrate our findings in numerical simulations. 

The paper is organized as follows. Section~\ref{sec:problem} introduces the considered control problem and the MPC algorithm. 
Stability properties of the closed loop system are analyzed in Section~\ref{sec:stability}. In Section~\ref{sec:kernel}, we show how models satisfying the required conditions can be learned using kernel interpolation.
Finally, numerical experiments are reported in Section~\ref{sec:numerics} and conclusions are drawn in Section~\ref{sec:conclusions}.
\\

\noindent\textbf{Notation}:
For a \karl{symmetric 
matrix~$M$, $\overline{\lambda}(M)$ and $\underline{\lambda}(M)$ denote its maximum and minimum eigenvalue, resp. $I_n\in\mathbb{R}^{n \times n}$, $0_{n \times m}\in\mathbb{R}^{n \times m}$ denote the identity and the zero matrix, resp. (dimensions are omitted when clear). 
$\mathrm{diag}$ denotes block diagonal matrices,
$\|\cdot\|$ the Euclidean norm, 
$\|x\|_Q^2 := x^\top Q x$ for a matrix $Q$ and a vector $x$.
For 
integers $a,b$, $a < b$, we abbreviate 
$[a:b] \coloneqq \mathbb{Z} \cap [a, b]$ and 
$x(b:a) \coloneqq [x(b)^\top, x(b-1)^\top, \dots, x(a)^\top]^\top$.
We write~$\pm a \coloneqq + a - a$}. 

\section{Problem formulation and MPC algorithm} \label{sec:problem}
The aim of the paper is to control a nonlinear system using a surrogate model inferred from input/output data.
The considered system is described by
\begin{align} \label{eq:NARX:sys}
    y(k + 1) &= F_\mathrm{y}(x(k), u(k)) \\
    x(k) &= [y(k:k-\nu+1)^\top, u(k - 1: k - \nu +1)^\top]^\top\nonumber 
\end{align}
where~$y(k) \in \R^p$ is the system output and $u(k) \in \mathbb{U}$ is its control input. 
$\mathbb{U} \subset \R^m$ is a compact  
set containing the origin in its interior, and~$x(k) \in \R^{n}$ for $n = \nu p+(\nu - 1)m$ is 
\karl{a vector consisting of inputs and outputs from the last $\nu$ steps.}
The parameter~$\nu \in \N$ defines how many of the past observations are required to 
characterize the dynamics, and corresponds to the lag of the system.
For sake of simplicity, we assume 
$x(k) \in \Omega \subseteq \R^n$, where $\Omega$ contains the origin in its interior and is a positive invariant set w.r.t.\ 
system~\eqref{eq:NARX:sys} under inputs $u \in \mathbb{U}$.\footnote{The case in which the invariance condition does not hold is studied in~\cite{schimperna2025data, schimperna2025stability}, in which a uniform error bound is \karl{leveraged 
to tighten the constraints and, then, to determine a} 
set in which stability holds.} 

We assume that $F_\mathrm{y}:\R^{n} \times \R^m\rightarrow \R^p$ is Lipschitz continuous w.r.t.\ its first argument \karl{uniformly in~$u$}, i.e., 
there exists a Lipschitz constant~$L_{F_\mathrm{y}} 
> 0$, such that 
\begin{align*}
    \|F_\mathrm{y}(x, u) - F_\mathrm{y}(x', u)\| \leq L_{F_\mathrm{y}}\|x - x'\| \qquad\forall\,u \in \mathbb{U}
\end{align*} 
for all $x, x' \in \Omega$. 
The system~\eqref{eq:NARX:sys} can also be written in an equivalent state-space form as
\begin{align}\label{eq:state:sys}
x(k + 1) &= F_\mathrm{x}(x(k), u(k)) \text{ with}  \\
    F_\mathrm{x}(x,u) &= \begin{bmatrix}
    F_\mathrm{y}(x, u) \\ 
    [I_{(\nu - 1)p}, 0_{(\nu - 1)p \times p+(\nu-1) m}]x \\
    u \\
    [0_{(\nu-2)m\times \nu p}, I_{(\nu-2)m}, 0_{(\nu-2)m\times m}]x
\end{bmatrix}\label{eq:F_x}
\end{align}
where $F_\mathrm{x}(x,u)$ has a Lipschitz constant $L_{F_\mathrm{x}}$. 
We assume that the origin is a controlled equilibrium of system~\eqref{eq:state:sys}, i.e. $F_\mathrm{y}(0, 0) = 0$ and $F_\mathrm{x}(0, 0) = 0$.

The MPC design is based on a data-driven surrogate model 
\begin{align}\label{eq:NARX:sur}
\begin{split}
    y(k+1) & \approx \F_\mathrm{y}(x(k), u(k)) \\
\end{split}
\end{align}
with right-hand side approximation~$\F_\mathrm{y}$ of $F_\mathrm{y}$, where the superscript~$\varepsilon \in (0, \bar{\varepsilon}]$, $\bar{\varepsilon} > 0$, stands for the approximation accuracy and is used in the following to indicate a dependence on the surrogate dynamics. 
In Section~\ref{sec:kernel-regression} we introduce kernel regression as a possible method to compute such a \karl{surrogate model in a data-driven fashion}. 
The surrogate system also has an equivalent state dynamics given by $\F_\mathrm{x}$, analogously defined to~\eqref{eq:F_x}, which is also assumed to render the set~$\Omega$ positive invariant to streamline the exposition.

The goal of the paper is to show that in presence of suitable bounds on the modeling error, MPC using the surrogate model~\eqref{eq:NARX:sur} in the optimization step stabilizes the \karl{controlled} system
~\eqref{eq:NARX:sys}. 
The MPC is designed considering the quadratic input/output stage cost
~$\ell : \R^p \times \R^m\rightarrow \R_{\geq 0}$ given by
\begin{equation} \label{eq:stagecosts}
    \ell(y, u) = \|y\|_Q^2 + \|u\|_R^2,
\end{equation}
where \karl{$Q 
\in \R^{p\times p}$ and  $R 
\in \R^{m \times m}$ are symmetric} positive-definite weighting matrices. 
The MPC cost function $\J_N:\Omega \times \mathbb{U}^N \to \R_{\geq 0}$ is given by 
\begin{subequations}
\begin{align}
    \J_N(\hat{x}, \mathbf{u}) & = \sum\nolimits_{k = 0}^{N-1} \ell (y^\varepsilon_\mathbf{u}(k+1; \hat{x}), u(k)) \\
    x^\varepsilon_\mathbf{u}(0; \hat{x}) & = \hat{x} \label{eq:mpc:initialization} \\
    x^\varepsilon_\mathbf{u}(k+1; \hat{x}) & = \F_\mathrm{x}(x^\varepsilon_\mathbf{u}(k; \hat{x}), u(k)), 
    k \in [0:\karl{N-1}] \label{eq:mpc:x_dynamics}\\
    y^\varepsilon_\mathbf{u}(k; \hat{x}) & = [I_p, 0_{\karl{p \times (n-p)}}] x^\varepsilon_\mathbf{u}(k; \hat{x}), 
    k \in [0:N] \label{eq:mpc:y_dynamics}
\end{align}
\end{subequations}
Finally, the MPC algorithm is reported in Algorithm~\ref{alg:MPC}. 
\begin{algorithm}[htb]
    \caption{Data-driven nonlinear MPC}\label{alg:MPC}
    \raggedright
    \smallskip\hrule
    \smallskip
    {\it Input}: Horizon $N \in \N$, 
    surrogate~$\F_\mathrm{x}$, stage cost $\ell$, \\
    \phantom{{\it Input}:} input constraints~$\mathbb{U}$ 
    \smallskip\hrule
    \medskip
    \textit{Initialization}: Set $k = 0$, and initialize the state as\\
    \phantom{\textit{Initialization}:} $\hat{x} = [y(0:-\nu)^\top, u(-1:-\nu)^\top]^\top$\\[2mm]
    \noindent\textit{(1)} If $k > 0$, measure output $y(k)$ and set\\
    \phantom{\textit{(1)}} $\hat{x}=x(k) = [y(k:k-\nu+1)^\top, u(k - 1: k - \nu +1)^\top]^\top$\\[1mm]
    \noindent\textit{(2)} Solve the optimal control problem~
    \begin{align}\label{eq:OCP}\tag{OCP}
    \begin{split}
        \V_N(\hat{x}) \coloneqq \min_{\mathbf{u}} 
        \quad & \J(\hat{x}, \mathbf{u}) \\
        \text{s.t.} \quad& \text{\eqref{eq:mpc:initialization}-\eqref{eq:mpc:x_dynamics}-\eqref{eq:mpc:y_dynamics}} \\
        & \karl{\mathbf{u} = \{ u(i) \}_{i=0}^{N-1} \subset 
        \mathbb{U}} 
    \end{split}
    \end{align}
    \hspace*{5mm} to obtain optimal control sequence $\karl{\mathbf{u}^\star =} \{u^\star(i)\}_{i = 0}^{N - 1}$ \\[1mm]
    \noindent\textit{(3)} Apply the MPC feedback law~$\mu_N^\varepsilon(\hat{x}) = u^\star({0})$ to the\\
    \hspace*{5mm} plant to generate the closed loop
    \begin{align*}
        y(k + 1) = F_\mathrm{y}(x(k), \mu_N^\varepsilon(x(k))),
    \end{align*}
    \hspace*{5mm} increment $k = k + 1$, and go to Step~(1).
    \smallskip\hrule
\end{algorithm}

\noindent For the closed-loop analysis, we also introduce the nominal cost function~$J_N$, and the nominal (optimal) value function~$V_N$, that are defined analogously to $\J_N$ and $\V_N$, but using the \karl{actual 
system dynamics~$F_\mathrm{x}$ instead of the surrogate~$\F_\mathrm{x}$ in~\eqref{eq:mpc:x_dynamics}. 
In case we have a linear surrogate model $\F_{\mathrm{x}}$, \eqref{eq:OCP} is a computationally efficient convex optimization problem, assuming also $\mathbb{U}$ is convex. In general, we consider a nonlinear true dynamics $F_{\mathrm{x}}$ and hence a nonlinear surrogate dynamics $\F_{\mathrm{x}}$ and hence \eqref{eq:OCP} is a nonlinear program, as standard in nonlinear MPC~\cite{RawlMayn17}}.

\section{Stability analysis} \label{sec:stability}

\karl{Our goal is to prove exponential stability of the data-driven MPC closed-loop system. To this end, we require the following properties of the model~$\F_\mathrm{y}$},
see, e.g., \cite{STRASSER2026112732}.
\begin{assumption} \label{asm:model:properties}
    For every $\varepsilon \in (0,\bar{\varepsilon}]$, \karl{$\bar{\varepsilon} > 0$}, let the surrogate model~\eqref{eq:NARX:sur} satisfy
    \begin{enumerate}
        \item \emph{proportional} error bounds
        \begin{equation}\label{eq:ass:error_bound}
            \|F_\mathrm{y}(x, u) - \F_\mathrm{y}(x, u)\| \leq c_x^\varepsilon\|x\| + c_u^\varepsilon\|u\|
        \end{equation} 
        for all $x \in \Omega$, $u \in \mathbb{U}$ with parameters $c_x^\varepsilon$and $c_u^\varepsilon$ satisfying $\lim_{\varepsilon \searrow 0} \max \{c_x^\varepsilon, c_u^\varepsilon \} = 0$.
        
        \item uniform Lipschitz continuity in the first argument on~$\Omega$, i.e., there exists $\bar{L}> 0$ such that, for every $\varepsilon \in (0,\bar{\varepsilon}]$, there is a Lipschitz constant $\LF$ with $\LF \leq \bar{L}$ satisfying, for all $x,x' \in \Omega$, $u \in \mathbb{U}$,
        \begin{align} \label{eq:ass:Lipschitz}
            \|\F_\mathrm{y}(x,u) - \F_\mathrm{y}(x',u)\| \leq \LF\|x-x'\|
        \end{align}
    \end{enumerate}
\end{assumption}
\karl{Assumption~\ref{asm:model:properties} can be rigorously verified for, e.g., kernel-based surrogate models as shown in Section~\ref{sec:kernel}}.

To derive the stability results, it is important to notice that the considered stage cost~\eqref{eq:stagecosts} only penalizes the system output, and 
is therefore only positive semi-definite in the state~$x$. 
Hence, the standard stability results for positive definite \karl{stage cost, 
cf.~\cite{grune2017nonlinear}}, cannot be applied directly. Instead, the stability proof relies on cost detectability~\cite{GrimMess05,KohlZeil23}.
In view of the NARX structure of the system~\eqref{eq:NARX:sys} and model~\eqref{eq:NARX:sur}, the following result\karl{, which is a straightforward adaptation of \cite[Remark~3]{KohlZeil23}}, establishes this detectability \karl{condition}.
\begin{proposition}\label{prop:cost_detect}
    (Cost detectability)\\ 
    \karl{For weighting matrix $P=\mathrm{diag}(Q,\dots,\frac{1}{\nu}Q,R,\dots, \frac{2}{\nu}R)$, the quadratic storage function $W(x)=\|x\|^2_{P}$ is given by}
    \begin{align}\label{eq:storage_equality}
        W(x(t))\hspace*{-0.5mm}=\hspace*{-1mm}\sum_{k=0}^{\nu-1}\textstyle\frac{\nu-k}{\nu}\displaystyle\|y(t-k)\|_Q^2 + \hspace*{-1mm} \sum_{k=1}^{\nu-1}\textstyle\frac{\nu-k+1}{\nu}\displaystyle\|u(t-k)\|_R^2.  
    \end{align}
    \karl{For any $x \in {\Omega}, u \in \mathbb{U}$,
    $x^+ = F_\mathrm{x}(x,u)$ and $y^+ = F_\mathrm{y}(x,u)$,
    the following inequality holds}:
    \begin{align}
    \label{eq:storage_inequality}
    W(x^+)\leq \dfrac{\nu-1}{\nu}W(x)+\ell(y^+,u).    
    \end{align}
\end{proposition}
\karl{This proposition verifies cost detectability with the storage function~$W$. In particular, recursive application of~\eqref{eq:storage_inequality} ensures that bounded cumulative costs ensure convergence of the state to zero. 
This connection will be crucial in the theoretical analysis to \karl{establish} 
stability of the MPC scheme, which minimizes the input-output cost $\ell$.
Note that Proposition~\ref{prop:cost_detect} holds for any system in NARX form of \eqref{eq:NARX:sys} including the surrogate models 
$\F_\mathrm{y}, \F_\mathrm{x}$}.

Since the MPC algorithm does not include terminal conditions, the stability proof utilizes the following cost controllability condition adapted from~\cite[Assumption~2]{KohlZeil23}. 
\begin{assumption}[Cost controllability] \label{asm:cost:controllability}
    System~\eqref{eq:NARX:sys} is \textit{cost controllable} with stage cost~\eqref{eq:stagecosts} on the set~$\Omega$, i.e., there exists a monotonically increasing bounded sequence $(B_{N})_{N \in \mathbb{N}_0}$ such that, for every $\hat{x} \in \Omega$ and every $N \in \mathbb{N}$, there exists a control sequence $\mathbf{u} \in \mathbb{U}^N$ satisfying the growth bound
    \begin{align}\label{eq:growthbound}
        V_{N}(\hat{x}) \leq J_{N}(\hat{x}, \mathbf{u}) \leq B_{N} \|\hat{x}\|^2.
    \end{align}
\end{assumption}
Cost controllability with a quadratic cost $\ell$ means that the system can be exponentially stabilized to the origin.
Compared to the positive-definite-cost case \cite{grune2017nonlinear}, we cannot have only the stage cost $\ell$ on the right hand side, see \cite{KohlZeil23} for an in-depth discussion.

In the following, we show that, if the system under control is cost controllable, then the same property is preserved by the surrogate model, and vice versa. This proposition follows \karl{a similar
reasoning like} 
\cite[Proposition 1]{schimperna2025data}. 
\karl{However, since the stage cost is only semi-definite, it cannot be used to bound the state prediction error, and the cost detectability function $W$ has to be used instead}.

\begin{proposition}\label{prop:cost:controllability}
    Let Assumptions~\ref{asm:model:properties} and~\ref{asm:cost:controllability} hold.
    Then, for given $\bar{N}$, the growth bound~\eqref{eq:growthbound} is satisfied for the surrogate model~\eqref{eq:NARX:sur} on~$\Omega$ for all $N \in [1:\bar{N}]$ uniformly in $\varepsilon$, i.e., there exists a monotonically increasing sequence $(B_{N}^\varepsilon)_{N \in [1:\bar{N}]}$, parametrized in~$\varepsilon$, such that, for each pair ${(\hat{x},N) \in \Omega 
    \times [1:\bar{N}]}$, there exists $\mathbf{u} \in \mathbb{U}^N$ satisfying
    \begin{align} \label{eq:growthbound:surrogate}
        \V_{N}(\hat{x}) \leq \J_{N}(\hat{x}, \mathbf{u}) \leq \B_{N} \|\hat{x}\|^2. 
    \end{align}
    Moreover, $\lim_{\varepsilon \searrow 0} \B_{N} = B_{N}$ for all $N \in [1:\bar{N}]$. 
    The statement holds also upon switching the roles of $F_\mathrm{y}$ and $\F_\mathrm{y}$, i.e., if the \karl{growth bound}
    ~\eqref{eq:growthbound} holds for the surrogate model $\F_\mathrm{y}$, then \karl{its counterpart}
    ~\eqref{eq:growthbound:surrogate} holds for the original system dynamics~$F_\mathrm{y}$ for all $N \in [1 : \bar{N}]$.
\end{proposition}
\begin{proof}    
    \ifbool{arxiv}{For sake of compactness, \karl{in the proof} we omit the dependence of all the state and output sequences from the initial condition $\hat{x}$.
    \karl{The aim of the proof is to study the difference between $J_N^\varepsilon(\hat{x}, \mathbf{u})$ and $J_N(\hat{x}, \mathbf{u})$ relying on bounds on $e_{\mathrm{y}}(k) := \| y_{\mathbf{u}}^\varepsilon(k) - y_{\mathbf{u}}(k) \|$.}
    \karl{By standard norm inequalities,} we have that
    \begin{align*}
         \J_N(\hat{x}, \mathbf{u}) &\leq J_N(\hat{x}, \mathbf{u}) \\
         &\hspace*{-2.5mm}+\hspace*{-0.5mm} \bar{\lambda}(Q) \Big[ \sum_{k = 0}^{N-1} e_\mathrm{y}(k+1)^2 \hspace*{-1mm}+ 2 e_\mathrm{y}(k+1)\|y_{\mathbf{u}}(k+1)\| \Big].
    \end{align*}
    Next, we derive bounds on $e_\mathrm{y}^2(k+1)$ and $e_\mathrm{y}(k+1)\|y_{\mathbf{u}}(k+1)\|$. 
    \karl{By} definition $e_\mathrm{y}(k)  \leq e_\mathrm{x}(k) := \| x_{\mathbf{u}}^\varepsilon(k) - x_{\mathbf{u}}(k) \|$ and 
    $\|F_\mathrm{x}(x, u) - \F_\mathrm{x}(x, u)\| = \|F_\mathrm{y}(x, u) - \F_\mathrm{y}(x, u)\|$, since the two functions only differ for the first $p$ components.
    Analogously to 
    \karl{the proof of \cite[Prop.~1]{schimperna2025data}}, we have that
    \begin{align}
        & e_\mathrm{y}(k + 1) \leq e_\mathrm{x}(k + 1) \nonumber \\
        = & \| \F_\mathrm{x}(x_{\mathbf{u}}^\varepsilon(k), u(k)) \pm F_\mathrm{x}(x_{\mathbf{u}}^\varepsilon(k), u(k)) - F_\mathrm{x}(x_{\mathbf{u}}(k), u(k))\| \nonumber \\
        \leq & c_x^\varepsilon \| x_{\mathbf{u}}^\varepsilon(k) \pm x_{\mathbf{u}}(k)\| + c_u^\varepsilon\|u(k)\| + L_{F_\mathrm{x}}\| x_{\mathbf{u}}^\varepsilon(k) - x_{\mathbf{u}}(k) \| \nonumber \\
        \leq & \bar{c}(\|x_{\mathbf{u}}(k)\| + \|u(k)\|) + (L_{F_\mathrm{x}} + c_x^\varepsilon) e_\mathrm{x}(k)\label{eq:cost:controllability:proof1:modified} \\
        \leq & \bar{c} \sum\nolimits_{i=0}^k (L_{F_\mathrm{x}}+c_x^\varepsilon)^{k-i} (\|x_{\mathbf{u}}(i)\| + \|u(i)\|) \label{eq:cost:controllability:proof2:modified}
    \end{align}
    with $\bar{c}:= \max\{c_x^\varepsilon, c_u^\varepsilon\}$. Let $d \coloneqq L_{F_\mathrm{x}} + c_x^\varepsilon$ and $\ell_\mathrm{y}(i) \coloneqq \ell(y_\mathbf{u}(i+1), u(i))$.
    Using Inequality~\eqref{eq:cost:controllability:proof1:modified} and the fact that $(a + b)^2 \leq 2 a^2 + 2 b^2$, we get 
    \begin{align}
        & e_\mathrm{y}(k+1)^2 \leq 
        4 \bar{c}^2 (\|x_{\mathbf{u}}(k)\|^2 + \|u(k)\|^2) + 2 d^2 e_\mathrm{x}(k)^2 \nonumber \\
        \leq \;& 4 \bar{c}^2  (\|x_\mathbf{u}(k)\|^2 + \underline{\lambda}(R)^{-1}\ell_\mathrm{y}(k)) + 2 d^2 e_\mathrm{x}(k)^2  \nonumber \\
        \leq \;& 4 \bar{c}^2  \sum_{i = 0}^{k} (2 d^2)^{k - i} (\|x_\mathbf{u}(i)\|^2 + \underline{\lambda}(R)^{-1}\ell_\mathrm{y}(i))
        \nonumber \\
        = \; & 4 \bar{c}^2 \sum_{j = 0}^{k} (2 d^2)^{j} (\|x_\mathbf{u}(k - j)\|^2 + \underline{\lambda}(R)^{-1}\ell_\mathrm{y}(k - j)). \label{eq:bound:ey:squared} 
    \end{align}
    Analogously, leveraging Inequality~\eqref{eq:cost:controllability:proof2:modified} and the fact that $2 \|a\| \|b\| \leq \|a\|^2 + \|b\|^2$ yields
    \begin{align}
        & e_\mathrm{y}(k) \|y_{\mathbf{u}}(k)\| \leq  e_\mathrm{x}(k)\|y_{\mathbf{u}}(k)\| \nonumber \\
        \leq & \bar{c} \sum_{i=0}^{k - 1} d^{k-1-i} \Big( \overbrace{\|x_{\mathbf{u}}(i)\| \|y_{\mathbf{u}}(k)\| + \|u(i)\| \|y_{\mathbf{u}}(k)\|}^{\leq \frac 12 ( \|x_{\mathbf{u}}(i)\|^2 + \|u(i)\|^2 + 2 \|y_{\mathbf{u}}(k)\|^2 )} \Big) \nonumber \\
        \leq & \frac{\bar{c}}{2} \sum_{i=0}^{k - 1} d^{k - 1 - i} \Big( \|x_{\mathbf{u}}(i)\|^2 + \underline{\lambda}(R)^{-1} \ell_\mathrm{y}(i) + 2 \|y_{\mathbf{u}}(k)\|^2 \Big) \nonumber \\
        = & \frac{\bar{c}}{2} \sum_{j = 0}^{k - 1} d^{j} \Big( \|x_{\mathbf{u}}(k - 1 -j)\|^2 + \underline{\lambda}(R)^{-1} \ell_\mathrm{y}(k - 1 -j) \nonumber \\
        &+ 2 \|y_{\mathbf{u}}(k)\|^2 \Big). \label{eq:bound:ey:times:y}
    \end{align}
    \karl{The second summand in \eqref{eq:bound:ey:squared} and the second and third summands in \eqref{eq:bound:ey:times:y} consist of terms that are included in the MPC cost function, and, thus, can be bounded by exploiting the cost controllability condition \eqref{eq:growthbound}, similarly to the proof of \cite[Prop.~1]{schimperna2025data}.}
    \karl{Instead, the term $\|x_{\mathbf{u}}(k- 1-j)\|^2$}
    cannot be included in $\ell_\mathrm{y}$ in view of the positive semi-definiteness of the cost.
    Hence, 
    \karl{this term is bounded exploiting the storage function $W$ introduced in Proposition~\ref{prop:cost_detect}}.
    In particular, for each $k \in [1:N-1]$ we can apply \eqref{eq:storage_equality} iteratively and use that $\frac{\nu - 1}{\nu} < 1$ and the assumed cost controllability to obtain
    \begin{align*}
        \|x_\mathbf{u}(k)\|_{P}^2 & \leq \frac{\nu - 1}{\nu} \|x_\mathbf{u}(k-1)\|_{P}^2 + \ell_\mathrm{y}(k-1) \\
        & \leq \sum\nolimits_{i=0}^{k-1} \ell_\mathrm{y}(i) + \|\hat{x}\|_{P}^2 \leq B_{k} \|\hat{x}\|^2 + \|\hat{x}\|_{P}^2,
    \end{align*}
    which, in view of the positive definiteness of $P$, implies 
    \begin{equation} \label{eq:bound:xu:squared}
        \|x_\mathbf{u}(k)\|^2 \leq \underline{\lambda}(P)^{-1} (B_k + \overline{\lambda}(P))\|\hat{x}\|^2.
    \end{equation}
    
    \karl{
    We can now study term $\sum_{k=0}^{N-1} e_\mathrm{y}(k)^2$. In view of \eqref{eq:bound:ey:squared} and \eqref{eq:bound:xu:squared}, we have that
    \begin{align*}
        & \sum_{k = 0}^{N-1} e_\mathrm{y}(k+1)^2 \\
        \leq& 4\bar{c}^{2} \sum_{k = 0}^{N-1} \sum_{j = 0}^{k} (2 d^2)^{j} \underline{\lambda}(P)^{-1} (B_{k-j} + \overline{\lambda}(P))\|\hat{x}\|^2 \\
        & +\sum_{k = 0}^{N - 1} \sum_{j = 0}^{k} \frac{(2 d^2)^{j}}{\underline{\lambda}(R)}\ell_\mathrm{y}(k - j) 
    \end{align*}
    Then, noting that the summation $\sum_{k = 0}^{N - 1} \sum_{j = 0}^{k}$ is equivalent to $\sum_{0 \leq j < k \leq N - 1} = \sum_{j = 0}^{N - 1} \sum_{k = j}^{N-1}$ and invoking the assumed cost controllability~\eqref{eq:growthbound}, we can bound the second summand in the previous inequality as follows:
    \begin{align*}
        & \sum_{k = 0}^{N - 1} \sum_{j = 0}^{k} (2 d^2)^{j}\frac{1}{\underline{\lambda}(R)}\ell_\mathrm{y}(k - j)  \\
        = & \sum_{j = 0}^{N - 1} (2 d^2)^{j}\underbrace{\sum_{k = j}^{N-1}\frac{1}{\underline{\lambda}(R)}\ell_\mathrm{y}(k - j)}_{= \underline{\lambda}(R)^{-1}\sum_{r = 0}^{N - j - 1}\ell_\mathrm{y}(r)} \leq \sum_{j = 0}^{N - 2} (2 d^2)^{j}\frac{B_{N - j}}{\underline{\lambda}(R)} \|\hat{x}\|^2.
    \end{align*}
    }
    \karl{
    We proceed similarly for $\sum_{k = 0}^{N - 1} e_\mathrm{y}(k+1) \|y_{\mathbf{u}}(k+1)\|$, starting from \eqref{eq:bound:ey:times:y} and applying \eqref{eq:bound:xu:squared}: 
    \begin{eqnarray*}
    & & \sum_{k = 0}^{N - 1} e_\mathrm{y}(k+1) \|y_{\mathbf{u}}(k+1)\| \\
    & \leq & \frac{\bar{c}}{2} \sum_{k = 0}^{N - 1} \sum_{j = 0}^{k} d^{j} \underline{\lambda}(P)^{-1} (B_{k-j} + \overline{\lambda}(P))\|\hat{x}\|^2 \\
    & & + \frac{\bar{c}}{2} \sum_{k = 0}^{N - 1} \sum_{j = 0}^{k} d^{j}\Big( \underline{\lambda}(R)^{-1} \ell(k- j) + 2 \|y_{\mathbf{u}}(k+1)\|^2 \Big).
    \end{eqnarray*}
    The terms in the last summation can be upper bounded by
    \begin{align*}
         \sum_{k = 0}^{N - 1} \sum_{j = 0}^{k} \frac{d^{j}}{\underline{\lambda}(R)} \ell_\mathrm{y}(k - j) &\leq \underline{\lambda}(R)^{-1} \sum_{j = 0}^{N - 1} d^{j}\underbrace{\sum_{k = j}^{N-1}\ell_\mathrm{y}(k - j)}_{\leq \sum_{r = 0}^{N-j-1} \ell_\mathrm{y}(r)} \\
         &\leq \underline{\lambda}(R)^{-1} \sum_{j = 0}^{N - 1} d^{j} B_{N-j}\|\hat{x}\|^2
     \end{align*}
     and by
    \begin{align*}
       & \sum_{k = 0}^{N - 1}\sum_{j = 0}^{k} d^{j} \|y_{\mathbf{u}}(k+1)\|^2 
       \leq \sum_{k = 0}^{N - 1}\sum_{j = 0}^{k} d^{j} \underline{\lambda}(Q)^{-1}\ell_\mathrm{y}(k) \\
       &\qquad = \underline{\lambda}(Q)^{-1}\sum_{k = 0}^{N - 1} \ell_\mathrm{y}(k) \sum_{j = 0}^{k-1} d^{j}\\
       &\qquad \leq \underline{\lambda}(Q)^{-1} \Bigg(\max_{k \in [0:N-1]}\sum_{j = 0}^{k-1}d^j \Bigg) \underbrace{\sum_{k = 0}^{N - 1}\ell_\mathrm{y}(k)}_{\leq B_N \|\hat{x}\|^2}. 
    \end{align*}
     }
    \karl{Combining the previous estimates leads to}
    \begin{align}\label{eq:B_N_epsilon}
        \J_N(\hat{x}, \mathbf{u}) &\leq \left(B_N + 4\bar{\lambda}(Q)\bar{c}^2 c_N + 2\bar{\lambda}(Q)\bar{c} \bar{c}_N \right)\|\hat{x}\|^2 \nonumber \\
        & =: B_N^\varepsilon \|\hat{x}\|^2,
    \end{align}
    where $c_N$ is given by
    \begin{align*}
        \sum_{k = 0}^{N - 1} \sum_{j = 0}^{k} (2 d^2)^{j} \underline{\lambda}(P)^{-1} (B_{k-j} + \overline{\lambda}(P)) + \sum_{j = 0}^{N - 1} (2 d^2)^{j} \frac{B_{N - j}}{\underline{\lambda}(R)} 
    \end{align*}
    and 
    \begin{align*}
        \bar{c}_N & \coloneqq \frac{1}{2} \Bigg(\sum_{k = 0}^{N - 1}\sum_{j = 0}^{k} d^{j} \underline{\lambda}(P)^{-1} (B_{k-j} + \overline{\lambda}(P) ) \\
    & + \sum_{j = 0}^{N-1}\underline{\lambda}(R)^{-1} d^{j} B_{N-j} + \underline{\lambda}(Q)^{-1} \hspace{-0.15cm}\max_{k \in [0:N-1]}\sum_{j = 0}^{k-1}d^jB_N\Bigg).
    \end{align*}
    In~\eqref{eq:B_N_epsilon}, for each $N \in [1:\bar{N}]$, $B_{N}^\varepsilon \rightarrow B_{N}$ for $\bar{c} := \max\{c_x^\varepsilon,c_u^\varepsilon\} \searrow 0$ as claimed.
    The symmetry of the statement can be proved with the same arguments of \cite{schimperna2025data}.}%
    {\karl{Since the proof follows similar steps as \cite[Prop.~1]{schimperna2025data}, we only provide a sketch and refer to the extended arXiv version, see \url{arxiv.org/abs/2603.16808}, 
    of our paper for a detailed proof.
    We leverage standard norm inequalities, the growth condition~\eqref{eq:growthbound} and the proportional error bound
    ~\eqref{eq:ass:error_bound} to bound the difference between $J_N^\varepsilon(\hat{x}, \mathbf{u})$ and $J_N(\hat{x}, \mathbf{u})$. The main technical extension relies on the storage function~$W$ from Proposition~\ref{prop:cost_detect}, which is applied iteratively to \karl{derive bounds on} the state~$\|x_\mathbf{u}(k)\|^2$, $k \in \{0, \dots, N - 1\}$, and enables us to deduce 
    an explicit formula for
    $B_N^\varepsilon$ in \eqref{eq:growthbound:surrogate}, where $B_N^\varepsilon \to B_n$ as $\varepsilon \to 0$. 
    }}
\end{proof}

Since the considered stage cost is positive semi definite, the stability proof needs to consider 
\begin{equation}\label{eq:Lyapunov:candidate}
    Y^\varepsilon_N \coloneqq V^\varepsilon_N+W
\end{equation}
as candidate Lyapunov function and cannot rely on the optimal value function~$\V_N$ only, see~\cite{GrimMess05}. \karl{Consequentely, \cite[Thm.1]{schimperna2025data} cannot be applied directly.
Hence, we first state the following theorem as an auxiliary result, which is a direct adaptation of \cite[Thm.~3]{KohlZeil23} to our setting. Theorem~\ref{thm:stability_nominal} establishes a relaxed Lyapunov inequality for the surrogate model~$\F_y$ assuming that a condition on the interplay of prediction horizon~$N$ and approximation accuracy holds. The verifyability of that condition will be the key challenge in showing our main result, i.e., Theorem~\ref{thm:AS}}. 
\begin{theorem}
\label{thm:stability_nominal}
    \karl{Let the growth bound~\eqref{eq:growthbound:surrogate} hold for the \eqref{eq:OCP} using the surrogate model~$\F_y$. Further, let the prediction horizon satisfy the inequality 
    \begin{align}\label{eq:N:choice}
     N>\underline{N}=1+\dfrac{\log(\bar{\gamma}^\varepsilon)-\log(\frac{1}{\nu})}{\log(1+\bar{\gamma}^\varepsilon)-\log(\bar{\gamma}^\varepsilon+\eta)}   
    \end{align}
    with 
    $\bar{\gamma}^\varepsilon = \frac {\max_N \B_N}{\sigma_{\min}(P)}$, $\eta=\frac{\nu-1}{\nu}$}. 
    Then, there exists $\alpha_N>0$ such that for all $x \in \Omega$ \karl{and $\Y_N$ given by~\eqref{eq:Lyapunov:candidate}}
    \begin{subequations}\label{eq:Lyap_decrease_nominal}
        \begin{align}
            W(x)\leq \Y_N(x)\leq (\bar{\gamma}+1)W(x),\\
            \Y_N(\F_\mathrm{x}(x,u)) \leq \Y_N(x) -\frac{\alpha_N}{\nu} \cdot W(x).
    \end{align}
    \end{subequations}
\end{theorem}
\karl{Condition~\eqref{eq:N:choice} holds for sufficiently large horizon~$N$ and implies $\alpha_N>0$.
Then, inequalities~\eqref{eq:Lyap_decrease_nominal} ensure that $\Y_N(x)$ is a Lyapunov function. 
Therein, 
$\alpha_N\in(0,1]$ is a suboptimality index that bounds the closed-loop performance compared the infinite-horizon optimal controller, cf.\ \cite[Thm.~1]{KohlZeil23}. One can show that $\alpha_N$ approaches one as the prediction horizon $N$ approaches infinity, 
i.e., can essentially recover the best possible performance using MPC}. 

In the following theorem we derive the main result of the paper, i.e., exponential stability of the data-driven MPC closed loop.  To do so, we rely on the error bounds of Assumption~\ref{asm:model:properties} to show that $\Y_N$ is also a Lyapunov function for the system~\eqref{eq:state:sys} under the control law $\mu_N^\varepsilon$, provided that~$\varepsilon$ is sufficiently small.
\begin{theorem}\label{thm:AS}
    Let the assumptions of Proposition~\ref{prop:cost:controllability} and of Theorem~\ref{thm:stability_nominal} hold.
    Then, there exists $\varepsilon_0 \in (0,\bar{\varepsilon}]$ such that the MPC controller of Algorithm~\ref{alg:MPC} ensures exponential stability of the origin for all $x(0)\in\Omega$ and for \karl{all} $\varepsilon \in (0,\varepsilon_0)$. 
\end{theorem}
\begin{proof}
    In Theorem~\ref{thm:stability_nominal}, we have shown the relaxed Lyapunov inequality
    \begin{align}\label{eq:rel-lyap-ineq}
        \Y_N(\F_\mathrm{x}(\hat{x},u)) - \Y_N(\hat{x}) \leq -\frac{\alpha_N}{\nu} \cdot W(\hat{x})
    \end{align}
    for the surrogate model~$\F_\mathrm{x}$.
    Considering the real system dynamics $F_\mathrm{x}$, we have that
    \begin{eqnarray}
        & & \Y_N(F_\mathrm{x}(\hat{x}, \mu_N^\varepsilon(\hat{x})))  
        \pm \Y_N(\F_\mathrm{x}(\hat{x}, \mu_N^\varepsilon(\hat{x}))) \nonumber \\
        & \leq & 
        \Y_N(\hat{x}) -\frac{1}{\nu} \alpha_N W(\hat{x}) \nonumber\\
        & & + \V_N(F_x(\hat{x}, \mu_N^\varepsilon(\hat{x}))) - \V_N(\F_\mathrm{x}(\hat{x}, \mu_N^\varepsilon(\hat{x}))) \label{eq:VN:error} \\
        & & + W(F_\mathrm{x}(\hat{x}, \mu_N^\varepsilon(\hat{x}))) - W(\F_\mathrm{x}(\hat{x}, \mu_N^\varepsilon(\hat{x}))). \label{eq:W:error}
    \end{eqnarray}
    Next, we exploit the proportional bounds of Assumption~\ref{asm:model:properties} to derive bounds for \eqref{eq:VN:error} and \eqref{eq:W:error}. \karl{While the term in \eqref{eq:VN:error} is analogous to what is obtained in the proof of \cite[Thm. 1]{schimperna2025data}, the term in \eqref{eq:W:error} is due to the considered input-output setting.}
    To this end, we recall that $\F_\mathrm{x}(\hat{x}, \mu_N^\varepsilon(\hat{x})) = x_{\mathbf{u}^\star}^\varepsilon(1; \hat{x})$, where $\mathbf{u}^\star$ is the optimal solution of \eqref{eq:OCP} with initial state $\hat{x}$, and we denote the real successor state of the system by $x^+ := F_\mathrm{x}(\hat{x}, \mu_N^\varepsilon(\hat{x}))$. 
    $\mathbf{u}^\sharp = (u^\sharp(i))_{i=0}^{N-1}$ represents the solution of the MPC optimization problem initialized with $\tilde{x}^+ := x_{\mathbf{u}^\star}^\varepsilon(1; \hat{x})$, which is never computed in the practice but needed to define $\V_N(\F_\mathrm{x}(\hat{x}, \mu^\varepsilon_N(\hat{x})))$.
    Moreover, we define 
    $y_{\mathbf{u}^\sharp}^\varepsilon(i; \tilde{x}^+) \coloneqq [I_p 0_{(n-p)\times p}]x_{\mathbf{u}^\sharp}^\varepsilon(i; \tilde{x}^+)$
    and $y_{\mathbf{u}^\sharp}^\varepsilon(i; x^+) \coloneqq [I_p 0_{(n-p)\times p}] y_{\mathbf{u}^\sharp}^\varepsilon(i; x^+)$
    for all $i \in [0:N]$.
    Then, optimality of MPC yields 
    \begin{align}\label{eq:VN:error:bound}
        & \underbrace{\V_N(F_\mathrm{x}(\hat{x}, \mu_N^\varepsilon(\hat{x})))}_{= \V_N(x^+) \leq J_N^\varepsilon(x^+, \mathbf{u}^\sharp)} - \V_N(\F_\mathrm{x}(\hat{x}, \mu_N^\varepsilon(\hat{x}))) \\
        \leq &\hspace*{-0.5mm} \sum_{i=0}^{N-1} \hspace*{-0.5mm} \Big( \ell(y_{\mathbf{u}^\sharp}^\varepsilon(i+1; x^+), u^\sharp(i)) - \ell(y_{\mathbf{u}^\sharp}^\varepsilon(i+1; \tilde{x}^+), u^\sharp(i)) \Big). \nonumber
    \end{align}
    Consider now the $i$-th term of this summation
    \begin{align}
        & \ell(y_{ \mathbf{u}^\sharp}^\varepsilon(i+1; x^+), u^\sharp(i)) - \ell(y_{\mathbf{u}^\sharp}^\varepsilon(i+1; \tilde{x}^+), u^\sharp(i)) \nonumber\\
        = & \|y_{\mathbf{u}^\sharp}^\varepsilon(i+1; x^+) \|_Q^2 - \| y_{\mathbf{u}^\sharp}^\varepsilon(i+1; \tilde{x}^+) \|_Q^2 \nonumber\\
        \leq & \|Q\| \| y_{ \mathbf{u}^\sharp}^\varepsilon(i+1; x^+) - y_{\mathbf{u}^\sharp}^{\varepsilon}(i+1; \tilde{x}^+ ) \| \nonumber \\
        & \cdot \| y_{\mathbf{u}^\sharp}^\varepsilon (i+1; x^+) \hspace*{-0.05cm}+\hspace*{-0.05cm} y_{\mathbf{u}^\sharp}^\varepsilon(i+1; \tilde{x}^+) \mp y_{\mathbf{u}^\sharp}^\varepsilon(i+1;\tilde{x}^+) \| \nonumber\\
        \leq & 2 \|Q\| \| y_{\mathbf{u}^\sharp}^\varepsilon(i+1;\tilde{x}^+) \| \| y_{\mathbf{u}^\sharp}^\varepsilon(i+1; \tilde{x}^+) - y_{\mathbf{u}^\sharp}^\varepsilon(i+1; x^+) \| \nonumber \\
        & + \|Q\| \| y_{\mathbf{u}^\sharp}^\varepsilon(i+1; \tilde{x}^+) - y_{ \mathbf{u}^\sharp}^\varepsilon(i+1; x^+) \|^2, \label{eq:diff-ell}
    \end{align}
    where we have used $\|a\|_M^2 - \|b\|_M^2 = (a+b)^\top M (a-b) \leq \|M\| \|a - b\| \|a + b\|$.
    
    In the following, we derive upper bounds for the terms $\| y_{\mathbf{u}^\sharp}^\varepsilon(i+1; \tilde{x}^+) - y_{\mathbf{u}^\sharp}^\varepsilon(i+1; x^+) \|$ and $\| y_{\mathbf{u}^\sharp}^\varepsilon(i+1; \tilde{x}^+) \|$. 
    First, we consider the term $\| y_{\mathbf{u}^\sharp}^\varepsilon(i+1; \tilde{x}^+) - y_{\mathbf{u}^\sharp}^\varepsilon(i+1; x^+) \|$, and we leverage \eqref{eq:ass:error_bound} and \eqref{eq:ass:Lipschitz} of Assumption~\ref{asm:model:properties} to infer
    \begin{eqnarray*}
        & & \| y_{\mathbf{u}^\sharp}^\varepsilon(i+1; \tilde{x}^+) - y_{\mathbf{u}^\sharp}^\varepsilon(i+1; x^+) \| \\
        &\leq & \|x_{\mathbf{u}^\sharp}^\varepsilon(i+1; \tilde{x}^+) - x_{\mathbf{u}^\sharp}^\varepsilon(i+1; x^+)\| \\
        & \leq & L_{\F_\mathrm{x}}^{i+1} \| x_{\mathbf{u}^\sharp}^\varepsilon(0; \tilde{x}^+) - x_{\mathbf{u}^\sharp}^\varepsilon(0; x^+) \| \\
        & \leq & L_{\F_\mathrm{x}}^{i+1} (c_x^\varepsilon \| \hat{x} \| + c_u^\varepsilon \|\mu_N^\varepsilon(\hat{x})\|).
    \end{eqnarray*}
    Second, we consider the term $\| y_{\mathbf{u}^\sharp}^\varepsilon(i+1; \tilde{x}^+) \|$. For the growth bound 
    of $\F_\mathrm{x}$ derived in Proposition~\ref{prop:cost:controllability} and for the relaxed Lyapunov inequality \eqref{eq:rel-lyap-ineq}, we have that
    \begin{align*}
        V_N^\varepsilon(\F_\mathrm{x}(\hat{x}, \mu_N^{\varepsilon}(\hat{x}))) & = \sum\nolimits_{i=0}^{N-1} \ell(y_{\mathbf{u}^\sharp}^\varepsilon(i+1; \tilde{x}^+),u^\sharp(i)) \\
        & \leq V_N^\varepsilon(\hat{x}) \leq B^\varepsilon_N \|\hat{x}\|^2.
    \end{align*}
    Since, we have the inequality $\ell(y_{\mathbf{u}^\sharp}^\varepsilon(i+1; \tilde{x}^+),u^\sharp(i)) \geq \| y_{\mathbf{u}^\sharp}^\varepsilon(i+1; \tilde{x}^+) \|_Q^2$ for all $i \in [0:N-1]$,
    we also have $\| y_{\mathbf{u}^\sharp}^\varepsilon(i+1; \tilde{x}^+) \|_Q^2 \leq B^\varepsilon_N \| \hat{x} \|^2$.
    By exploiting standard inequalities of weighted squared norms and taking the square root, we have $\| y_{\mathbf{u}^\sharp}^\varepsilon(i+1; \tilde{x}^+) \| \leq \tilde{c} \| \hat{x} \|$, where $\tilde{c} \coloneqq \sqrt{B^\varepsilon_N/\underline{\lambda}(Q)}$.
    Then, substituting these inequalities in~\eqref{eq:diff-ell},
    we get
    \begin{eqnarray} \label{eq:ell:error}
         & & \ell(y_{\mathbf{u}^\sharp}^\varepsilon(i+1; x^+), u^\varepsilon(i)) - \ell(y_{\mathbf{u}^\sharp}^\varepsilon(i+1; \tilde{x}^+), u^\sharp(i)) \nonumber \\
         & \leq & 2 \|Q\| \tilde{c}  L_{\F_\mathrm{x}}^{i+1} (c_x^\varepsilon \|\hat{x}\|^2 + \underbrace{c_u^\varepsilon \|\hat{x}\| \|\mu_N^\varepsilon(\hat{x})\|}_{\leq \frac{1}{2} c_u^\varepsilon (\|\hat{x}\|^2 + \|\mu_N^\varepsilon(\hat{x})\|^2)}) \nonumber \\
         & & + 2 \|Q\|  L_{\F_\mathrm{x}}^{2(i+1)} \Big( (c_x^\varepsilon)^2 \|\hat{x}\|^2 + (c_u^\varepsilon)^2 \|\mu_N^\varepsilon(\hat{x})\|^2 \Big).
    \end{eqnarray}
    The terms in \eqref{eq:W:error} can be studied with a similar reasoning by noting that
    \begin{align*}
       & W(F_\mathrm{x}(\hat{x}, \mu_N^\varepsilon(\hat{x}))) - W(\F_\mathrm{x}(\hat{x}, \mu_N^\varepsilon(\hat{x}))) \\
       = & W([F_\mathrm{y}(\hat{x}, \mu_N^\varepsilon(\hat{x}))^\top\hspace*{-1mm}, ([I, 0]\hat{x})^\top\hspace*{-1mm}, \mu_N^\varepsilon(\hat{x})^\top\hspace*{-1mm}, ([0, I, 0] \hat{x})^\top]^\top) \\
       & - W([\F_\mathrm{y}(\hat{x}, \mu_N^\varepsilon(\hat{x}))^\top \hspace*{-1mm}, ([I, 0]\hat{x})^\top \hspace*{-1mm}, \mu_N^\varepsilon(\hat{x})^\top \hspace*{-1mm}, ([0, I, 0] \hat{x})^\top]^\top) \\
       =& \| F_\mathrm{y}(\hat{x}, \mu_N^\varepsilon(\hat{x})) \|_Q^2 - \|\F_\mathrm{y}(\hat{x}, \mu_N^\varepsilon(\hat{x}))\|_Q^2 \\
       = & \|y_{\mathbf{u}^\sharp}^\varepsilon(0; x^+) \|_Q^2 - \| y_{\mathbf{u}^\sharp}^\varepsilon(0; \tilde{x}^+) \|_Q^2 \nonumber\\
       \leq & 2 \|Q\| \tilde{c} (c_x^\varepsilon \|\hat{x}\|^2 + \frac{1}{2} c_u^\varepsilon (\|\hat{x}\|^2 + \|\mu_N^\varepsilon(\hat{x})\|^2)) \nonumber \\
       & + 2 \|Q\| \Big( (c_x^\varepsilon)^2 \|\hat{x}\|^2 + (c_u^\varepsilon)^2 \|\mu_N^\varepsilon(\hat{x})\|^2 \Big),
    \end{align*}
    where the last bound is obtained applying the same reasoning we used for \eqref{eq:diff-ell} considering $i+1=0$. 
    
    Substituting this back in \eqref{eq:VN:error}, \eqref{eq:W:error} and \eqref{eq:VN:error:bound}, we obtain
    \begin{eqnarray*}
        & & Y_N^\varepsilon(F_\mathrm{x}(\hat{x}, \mu_N^{\varepsilon}(\hat{x}))) \\
        & \leq & Y_N^\varepsilon(\hat{x}) -\frac{1}{\nu} \alpha_N W(\hat{x}) + C_x \| \hat{x} \|^2 + C_u \| \mu_N^\varepsilon(\hat{x}) \|^2,
    \end{eqnarray*}
    where $C_x := 2 \|Q\| \sum_{i=0}^{N} \big( \tilde{c} L_{\F_\mathrm{x}}^{i} (c_x^\varepsilon + \frac{1}{2} c_u^\varepsilon) + ( L_{\F_\mathrm{x}}^{i} c_x^\varepsilon)^2 \big)$ and $C_u := 2  \|Q\| \sum_{i=0}^{N} \big( \tilde{c} L_{\F_\mathrm{x}}^{i} \frac{1}{2} c_u^\varepsilon +  ( L_{\F_\mathrm{x}}^{i} c_u^\varepsilon)^2 \big)$.
    The values of $C_x$ and $C_u$ can be made arbitrarily small with sufficiently small proportionality constants $c_x^\varepsilon$ and $c_u^\varepsilon$. 
    Moreover, in view of the results of Proposition~\ref{prop:cost:controllability}, we have that $\|\mu_N^\varepsilon(\hat{x})\|_R^2 \leq \V_N(\hat{x}) \leq B_N^\varepsilon \|\hat{x}\|^2$,
    which implies that $\|\mu_N^\varepsilon(\hat{x})\| \leq \sqrt{B_N^\varepsilon / \underline{\lambda}(R)} \|\hat{x}\|$.
    Then, there exists a sufficiently small $\varepsilon_0$ 
    such that $c^\varepsilon_x$, $c^\varepsilon_u$ are sufficiently small to ensure the inequality
    \[
        -\frac{1}{\nu} \alpha_N W(\hat{x}) + C_x \| \hat{x} \|^2 + C_u \| \mu_N^\varepsilon(\hat{x}) \|^2 < -\karl{\frac{1}{\nu}}\bar{\alpha} \karl{W(\hat{x})} 
    \]
    for 
    $\bar{\alpha} \in \karl{(0,\alpha_N)}$ and for all $\varepsilon \in (0,\varepsilon_0)$. 
    This completes the proof and, thus, shows exponential stability of the MPC closed loop based on the surrogate model~\eqref{eq:NARX:sur}.
\end{proof}

\karl{The proof of Theorem~\ref{thm:AS} shows that for any desired suboptimality index $\bar{\alpha} \in (0, \alpha_N)$, there exist a sufficiently small $\varepsilon_0$ such that the data-driven MPC closed-loop converges with the corresponding rate. In particular, the convergence rate of the data-driven MPC approaches the convergence rate of the MPC with an exact prediction model as the approximation error tends to zero}.

\section{Kernel interpolation: Data-driven models} \label{sec:kernel}
\label{sec:kernel-regression}
In the following, we show how we can learn a function $\F_\mathrm{y}$ satisfying Assumption~\ref{asm:model:properties} from input-output data using kernel interpolation~\cite{wendland2004scattered}.

Specifically, we have a data set $\mathcal{X}$ consisting of $\xi_i=(x_i,u_i)\in\Omega\times \mathbb{U}=:\Omega_{\xi}\subseteq\mathbb{R}^{n+m}$ \karl{and} 
$y_i=F_{\mathrm{y}}(\xi_i)$, $i \in [1:D]$. 
Since we can identify each component independently, we focus on estimating $\F_{\mathrm{y}}$ for a scalar output ($p=1$) to simplify the exposition.
We denote the fill distance of this data by
\begin{align*}
h_\mathcal{X} \coloneqq \sup_{(x, u) \in \Omega\times \mathbb{U}}\min_{(x_i,u_i)\in \mathcal{X}}\|(x, u)-(x_i, u_i)\|.
\end{align*}
Suppose that this data set contains the origin, i.e.,  $0\in\mathcal{X}$.
Let $\k:\Omega_\xi\times\Omega_\xi\rightarrow\mathbb{R}_{\geq 0}$ be a symmetric, strictly positive kernel with corresponding reproducing kernel Hilbert space (RKHS) denoted by $\mathbb{H}$. Furthermore, suppose that $F_{\mathrm{y}} \in \mathbb{H}$. 
Kernel interpolation yields the unique function that interpolates the data with the minimal RKHS norm, which is given by
\begin{align}
    \F_{\mathrm{y}}(\xi) = \k_{\mathcal{X}}(\xi)^\top K_{\mathcal{X}}^{-1} Y,
\end{align}
with $\k_{\mathcal{X}}:\Omega_\xi\rightarrow\mathbb{R}^{D}$, $\k_{\mathcal{X}}(\xi) = (\k(\xi, \xi_i))_{i = 1}^D$, kernel matrix $K_\mathcal{X} = (\k(\xi_i, \xi_j))_{i, j = 1}^D\in\mathbb{R}^{D\times D}$, and $Y=(y_i)_{i=1}^D\in\mathbb{R}^{D}$. 
Kernel interpolation enjoys the following error bound (cf. \cite[Sec.~14.1]{fasshauer2015kernel})
\begin{align}
    |\F_{\mathrm{y}}(\xi)-F_{\mathrm{y}}(\xi)|\leq P(\xi)\|    F_{\mathrm{y}}\|_{\mathbb{H}}\qquad \forall\,\xi\in\Omega
\end{align}
with the power function $P:\Omega_\xi\rightarrow\mathbb{R}_{\geq 0}$,
\begin{align*}
    P(\xi)=\sqrt{\k(\xi,\xi)-\k_{\mathcal{X}}(\xi)^\top K_{\mathcal{X}}^{-1}\k_{\mathcal{X}}(\xi)}.
\end{align*}
Suppose the RKHS is norm equivalent to the Sobolev space
of order $s$, e.g., by choosing a corresponding Matern or Wendland kernel~$\k$ of sufficient smoothness. 
Then, according to~\cite[Thm.~5.4]{kanagawa2018gaussian}, the 
power function satisfies 
\begin{align*}
    P(\xi)\leq C h_{\mathcal{X}}^{s-(n+m)/2-1}\|\xi\|
\end{align*}
for some constant $C>0$, where we use the fact that $0\in\mathcal{X}$, i.e., the origin is contained in the data set. 
By setting $s>1+(n+m)/2$, we satisfy the proportional error bounds with $\varepsilon$ given by the fill distance~$h_{\mathcal{X}}$. 
Furthermore, if kernel $\k$ is twice continuous differentiable with a bounded Hessian, then both functions $\F_{\mathrm{y}},F_{\mathrm{y}}\in\mathbb{H}$ are also Lipschitz continuous~\cite{fiedler2023lipschitz}. 

In conclusion, Assumption~\ref{asm:model:properties} holds by using kernel interpolation if: (i) the data has a small enough fill distance $h_{\mathcal{X}}\rightarrow 0$, (ii) the equilibrium at the origin is contained in the data set, and (iii) the unknown function $F_{\mathrm{y}}$ lies in the RKHS~$\mathbb{H}$ with a suitably chosen kernel~$\k$.

\section{Numerical example} \label{sec:numerics} 
The exponential stability of the data-driven MPC is illustrated in a two-tank example~\cite{bistak2014model} described by
\begin{align}
\begin{split}\label{eq:ex:two_tanks}
    \dot{h}_1 &= c_{1,2} \sqrt{h_2 - h_1} + c_{2}\sqrt{h_1} \\ 
    \dot{h}_2 &= \frac{1}{A_1}u - c_{1,2}\sqrt{h_2-h_1}
\end{split}
\end{align}
with constants~$A_1 = 0.001, c_{1, 2} = 0.0254, c_{2} = 0.0261$. The system~\eqref{eq:ex:two_tanks} is integrated by using the classical fourth-order Runge–Kutta method (RK4) using a sampling time~$\Delta t = 10$s.
The output of the system is given by $y = h_1$, and the control objective is to steer the system to the equilibrium point $\bar{h} = (0.0438, 0.09)^\top$, $\bar{u} = 5.461 \cdot 10^{-6}$, while respecting the input constraint $u \in \mathbb{U} \coloneqq [3.16 \cdot 10^{-6}, 4.76 \cdot 10^{-5}]$.

The MPC algorithm is based on a surrogate model obtained via kernel interpolation, using the Wendland kernel function $\k(x, x') = \phi(\|x - x'\|)$, with $\phi(r) \coloneqq \frac{1}{30}(1-r)^5(5r + 1)$ for $r \in [0,1]$ and 0 otherwise.
The lag of the system is $\nu = 2$, leading to a model with a state $x \in \Omega \subseteq \R^3$. We consider the domain $\Omega = [0, 0.5]^2 \times \mathbb{U}$. 
To ensure proportional error bounds, the first data point is in the reference equilibrium, i.e. $x_1 = [\bar{h}_1, \bar{h}_1, \bar{u}]$, $u_1 = \bar{u}$ and $y_1 = \bar{h}_1$.
For the simulations, we consider datasets with \karl{$D \in \{100, 500, 2500\}$} data points.
In order to obtain models with an equilibrium point in the origin, the input and output data are shifted with respect to their reference and rescaled before the model identification.

For MPC, we consider a prediction horizon $N=20$ as well as weights $Q = 1$ and $R = 10^{-1}$. 
\karl{As comparison strategy, we consider the nominal MPC which uses the exact model \eqref{eq:ex:two_tanks} for prediction. The nominal MPC is implemented with the same cost function and prediction horizon of the data-driven MPC.}
\begin{figure}
    \vspace{2mm}
    \centering
    \includegraphics[width=0.7\linewidth]{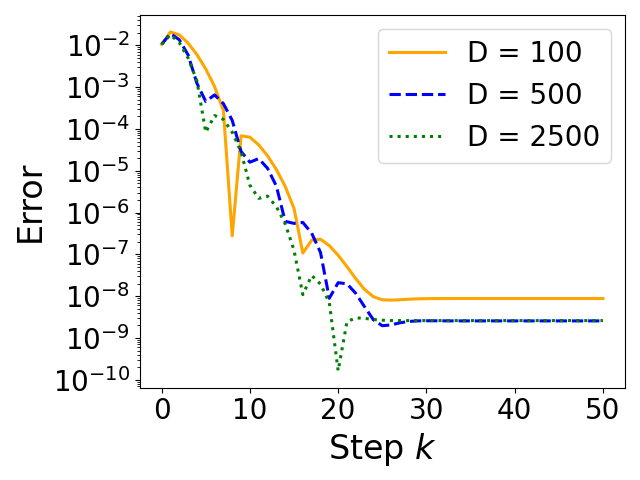}
    \caption{\karl{Output errors $\|h_1(k) - \bar{h}_1\|$ 
    of the data-driven MPC and the nominal MPC closed-loops for $k \in [0:50]$, where the models are obtained with $D \in \{100, 500, 2500\}$ data points}.
    }
    \label{fig:1}
\end{figure}
  
\begin{figure}[htbp]
  \centering
  
    \includegraphics[width=0.8\linewidth]{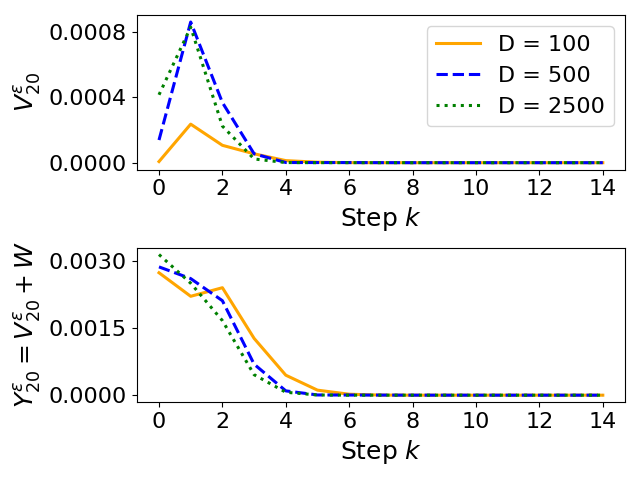}

  \caption{\karl{Optimal value function~$V^\varepsilon_N$ 
  and Lyapunov function~$Y^\varepsilon_N$ 
  (bottom) for horizon~$N = 20$ and where the models are obtained with $D \in \{100, 500, 2500\}$ data points}.} 
  \label{fig:2}
\end{figure}
The simulation results are reported in \karl{Figs.~\ref{fig:1} and \ref{fig:2}}. \karl{Fig.~\ref{fig:1} illustrates the tracking errors~$\|h_1(k) - \bar{h}_1\|$ of the data-driven MPC closed-loop, comparing the models computed with different numbers of data points, and of the nominal MPC closed-loop. It can be observed that the 
convergence rate is faster with a larger dataset, even if it does not reach the convergence speed of the closed-loop obtained with perfect system knowledge.
Moreover, Fig.~\ref{fig:2} shows the corresponding trajectories of the optimal value function~$V_N^\varepsilon$ and the Lyapunov function candidate~$Y_N^\varepsilon$. 
It is evident that $V_N^\varepsilon$ fails to serve as a Lyapunov function for any of the surrogate models. When considering $Y_N^\varepsilon$, we observe similar non-monotonic behaviour for $D = 100$ as seen with the optimal value function, however, with an increasing number of data points, i.e., $D \in \{500, 2500\}$, $Y_N^\varepsilon$ is monotonically decreasing. This verifies the findings of Theorem~\ref{thm:AS}, indicating that a sufficiently small modeling error ensures exponential stability. 
} 


\section{Conclusions} \label{sec:conclusions}

In this paper, we have provided sufficient conditions such that data-driven MPC without terminal conditions ensures exponential stability for nonlinear input-output systems. 
The key requirement is \karl{the combination of cost detectability and 
a proportional error bound. The latter 
can be achieved using kernel interpolation. Future work could consider investigating other data-driven models for proportional error bounds. 
Moreover, an} 
interesting open research direction is the inclusion of general noise in this framework using Gaussian process models~\cite{scampicchio2025gaussian}.

\bibliographystyle{ieeetr}
\bibliography{refs_input_output}

\end{document}